\documentclass[a4paper,11pt]{amsart}
\usepackage{amssymb,amsmath,amsthm}
\usepackage{fullpage}
\usepackage{color}
\definecolor{darkblue}{rgb}{0.2,0.2,0.6}
\definecolor{darkred}{rgb}{0.6,0.1,0.1}
\usepackage[colorlinks=true,linkcolor=darkblue, citecolor=darkred]{hyperref}
\usepackage{eucal, lmodern}
\usepackage{stackrel, enumerate}
\usepackage{graphicx, bbm}
\newcommand{\eg}{\emph{e.g.}}
\newcommand{\cf}{\emph{cf.}}

\newcommand\nb{\nabla}
\newcommand{\beq}{\begin{equation} \begin{split}}
\newcommand{\eeq}{\end{split} \end{equation}}

\newcommand\Omg{\Omega}
\newcommand\Omgc{\Omega^{\rm c}}

\renewcommand\and{\qquad\text{and}\qquad}

\newcommand\sm{\setminus}

\newcommand{\comm}[1]{}

\def\sfH{\mathsf{H}}

\def\bm1{\mathbbm{1}}
\def\G{\Gamma}

\def\s{\sigma}

\def\p{\partial}

\def\omg{\omega}

\def\arr{\rightarrow}

\newcommand{\sfJ}{\mathsf{J}}

\renewcommand{\gg}{{\gamma}}

\def\tt{\theta}
\def\aa{\alpha}
\def\lm{\lambda}

\def\s{\sigma}

\def\sess{\sigma_{\rm ess}}

\def\ii{{\mathsf{i}}}
\def\p{\partial}

\def\kp{\kappa}

\def\sfH{\mathsf{H}}

\def\dd{{\,\mathrm{d}}}
\def\der{{\mathrm{d}}}

\def\omg{\omega}

\def\sfU{\mathsf{U}}

\newcounter{counter_a}
\newenvironment{myenum}{\begin{list}{{\rm(\roman{counter_a})}}%
{\usecounter{counter_a}
\setlength{\itemsep}{1.ex}\setlength{\topsep}{0.8ex}
\setlength{\leftmargin}{5ex}\setlength{\labelwidth}{5ex}}}{\end{list}}

\usepackage[latin1]{inputenc}
\usepackage[T1]{fontenc}

\numberwithin{figure}{section}
\numberwithin{equation}{section}
\theoremstyle{plain}
\newtheorem*{thm*}{Theorem}
\newtheorem{thm}{Theorem}[section]

\newtheorem{prop}[thm]{Proposition}

\newtheorem{cor}[thm]{Corollary}
\newtheorem{conj}[thm]{Conjecture}

\newtheorem{dfn}[thm]{Definition}
\theoremstyle{remark}
\newtheorem{remark}[thm]{Remark}
\theoremstyle{plain}

\newcommand{\beu}{\begin{equation*}}
\newcommand{\eeu}{\end{equation*}}
\newcommand{\besu}{\begin{equation*}
\begin{aligned}}
\newcommand{\eesu}{\end{aligned}
\end{equation*}}
\newcommand{\bes}{\begin{equation}
\begin{aligned}}
\newcommand{\ees}{\end{aligned}
\end{equation}}

\newcommand\cBc{{\mathcal B}^{\rm c}}

\newcommand\cB{\mathcal B}

\newcommand\cL{\mathcal L}

\newcommand\frh{\mathfrak h}

\newcommand\ov{\overline}

\DeclareMathOperator\dom{dom}
\DeclareMathOperator\ran{ran}

\newcommand\void[1]{}

\def\ov{\overline}

\def\ran{{\rm ran\,}}

      \def\dC{{\mathbb C}}

   \def\dN{{\mathbb N}}   
      \def\dR{{\mathbb R}}

   \def\dZ{{\mathbb Z}}

   \def\sfH{{\mathsf H}}   
\def\sfJ{{\mathsf J}}

      \def\sfU{{\mathsf U}}

   \def\cB{{\mathcal B}}

      \def\cL{{\mathcal L}}
      \def\cO{{\mathcal O}}

\usepackage[left=2.7cm, right=2.7cm, marginparwidth=2cm, textheight = 24cm ]{geometry}
\usepackage[normalem]{ulem}
\definecolor{DarkGreen}{rgb}{0,0.5,0.1}

\definecolor{DarkBlue}{rgb}{0,0.1,0.5}

\newcommand\soutD{\bgroup\markoverwith
	{\textcolor{DarkGreen}{\rule[.5ex]{2pt}{1pt}}}\ULon}
\newcommand\soutP{\bgroup\markoverwith
	{\textcolor{blue}{\rule[.5ex]{2pt}{1pt}}}\ULon}
\newcommand{\Hm}[1]{\leavevmode{\marginpar{\tiny%
			$\hbox to 0mm{\hspace*{-0.5mm}$\leftarrow$\hss}%
			\vcenter{\vrule depth 0.1mm height 0.1mm width \the\marginparwidth}%
			\hbox to
			0mm{\hss$\rightarrow$\hspace*{-0.5mm}}$\\
			\relax\raggedright #1}}}

\begin{document}

\title[Optimisation and monotonicity of the second Robin eigenvalue on a planar exterior domain]{Optimisation and monotonicity of the second Robin eigenvalue on a planar exterior domain}

\author[D.~Krej\v{c}i\v{r}\'{i}k]{David Krej\v{c}i\v{r}\'{i}k}
\address{(D.~Krej\v{c}i\v{r}\'{i}k) Department of Mathematics\\ Faculty of Nuclear Sciences and Physical
	Engineering\\
	Czech Technical University in Prague\\ Trojanova 13, 120 00, Prague, Czech
	Republic 
}
\email{david.krejcirik@fjfi.cvut.cz}
\author[V.~Lotoreichik]{Vladimir Lotoreichik}
\address{(V.~Lotoreichik)
	Department of Theoretical Physics\\
	Nuclear Physics Institute, Czech Academy of Sciences, 
	25068 \v{R}e\v{z}, Czech Republic
}
\email{lotoreichik@ujf.cas.cz}

\begin{abstract}
We consider the Laplace operator in the exterior of a compact set in the plane,
subject to Robin boundary conditions.  
If the boundary coupling is sufficiently negative, 
there are at least two discrete eigenvalues below the essential spectrum.
We state a general conjecture that the second eigenvalue 
is maximised by the exterior of a disk under 
isochoric or isoperimetric constraints.
We prove an isoelastic version of the conjecture 
for the exterior of convex domains.
Finally, we establish a monotonicity result for the second eigenvalue
under the condition that the compact set is strictly star-shaped 
and centrally symmetric.
\end{abstract}


\maketitle

\section{Introduction}
\subsection{Motivations}
Spectrally optimal shapes are traditionally studied for bounded domains.
One motive is the classical interpretation of eigenvalues 
in terms of resonant frequencies of vibrating systems. 
Another reason is the necessity to ensure the very existence 
of the eigenvalue to be optimised. 
Indeed, taking the Dirichlet (or Neumann)
Laplacian as the traditional differential operator
whose eigenvalues should be optimised,
it is well known that
its spectrum is purely discrete for all (sufficiently regular) bounded sets. 

The development of mesoscopic physics has brought unprecedented motivations to study 
spectral properties of unbounded domains within the framework of quantum theory.
Then the desired playground 
for analogues of classical spectral isoperimetric-type inequalities 
is the exterior of a compact set. 
However, it is \emph{a priori} not clear whether eigenvalues exist there.
In fact, there is no discrete spectrum 
for the Dirichlet and Neumann boundary conditions.

The present authors remarked previously in \cite{KL18,KL20}
that attractive Robin boundary conditions lead to equally interesting 
spectral-optimisation problems in the exterior. 
The observation is based on the fact that 
given any bounded smooth planar domain~$\Omega$   
and denoting by~$\Omgc :=\dR^2\sm\ov\Omg$ its exterior,
the spectral threshold

\begin{equation}\label{Rayleigh}
  \lambda_1^\alpha(\Omgc)
  :=
  \inf_{ u \in H^1(\Omgc) \setminus \{0\} }
  \frac{\displaystyle \int_{\Omgc}|\nb u|^2 
  +\alpha\int_{\p\Omg}|u|^2}{\displaystyle\int_{\Omgc}|u|^2}
\end{equation}
represents the first (negative) eigenvalue of the Robin Laplacian 
in $L^2(\Omgc)$ whenever~$\alpha$ is negative.
Keeping~$\alpha$ negative and enlarging~$|\alpha|$,
one can produce an arbitrary number of eigenvalues
below the essential spectrum.
Establishing brand-new spectral isoperimetric-type inequalities 
for the eigenvalues constitutes an interesting direction of 
current research in spectral geometry.

The classical interpretation of Robin boundary conditions with positive~$\alpha$
for a bounded domain is that of a vibrating membrane with edges attached by
springs of positive friction (the limit $\alpha \to \infty$ realises
Dirichlet boundary conditions corresponding to fixed edges). 
Negative values of~$\alpha$ can be interpreted as a supply of energy
at the boundaries.

We rather rely on a quantum-mechanical interpretation for which
the Robin Laplacian is the Hamiltonian of a quantum particle 
constrained to a possibly unbounded domain (modelling a nanostructure), 
subject to a strongly localised potential at the boundary. 
Positive and negative values of~$\alpha$ correspond to 
repulsive and attractive interactions, respectively.
Eigenvalues then represent bound-state energies of the particle,
associated with stationary solutions of the Schr\"odinger equation.

\subsection{The first eigenvalue}
The main spectral-optimisation results of \cite{KL18,KL20} 
in the planar case can be summarised as follows.

\begin{thm}[\cite{KL18,KL20}]\label{Thm.lambda1}
If~$\Omega$ is simply connected, then
\begin{enumerate}
\item[\emph{(a)}]
$ 
  \lambda_1^\alpha(\Omgc) \leq \lambda_1^\alpha(\cBc)
$
where~$\cB$ is the disk of the same area as~$\Omega$;
 \item[\emph{(b)}]
$ 
  \lambda_1^\alpha(\Omgc) \leq \lambda_1^\alpha(\cBc)
$
where~$\cB$ is the disk of the same perimeter as~$\Omega$.
\end{enumerate}
\end{thm}

Here the statements are valid for every real~$\alpha$,
but they are trivial for non-negative~$\alpha$
when $\lambda_1^\alpha(\Omgc) = 0$ represents the lowest
point in the essential spectrum of the Robin Laplacian. 

Property~(a) is an exterior variant of the Faber--Krahn inequality.    
More specifically, the latter states
$ 
  \lambda_1^{\rm D}(\Omega) \geq \lambda_1^{\rm D}(\cB)
$,
where the first Dirichlet eigenvalue $\lambda_1^{\rm D}(\Omega)$ 
is defined as in~\eqref{Rayleigh} but with trial functions~$u$
from $H_0^1(\Omega)$ and with the boundary term absent
(formally, $\alpha = \infty$).  
This celebrated inequality was conjectured 
by Lord Rayleigh in 1877 \cite{Rayleigh}
and proved independently 
by Faber in 1923~\cite{Faber} 
and Krahn in 1924~\cite{Krahn}.
The Faber--Krahn inequality extends to 
$ 
  \lambda_1^\alpha(\Omega) \geq \lambda_1^\alpha(\cB)
$
for all positive~$\alpha$, but this result is much more recent.
It was established by Bossel in 1986 \cite{Bossel_1986}
in the present planar setting
and by Daners in 2006 \cite{Daners_2006} for higher dimensions. 
We also refer to 
\cite{Bucur-Daners_2010,Bucur-Giacomini_2015,
Bucur-Ferone-Nitsch-Trombetti_2018,
Alvino-Nitsch-Trombetti,
Alvino-Chiacchio-Nitsch-Trombetti_2021,
Alvino-Chiacchio-Nitsch-Trombetti}
for alternative approaches and most recent developments.

The present setting of negative~$\alpha$ is even more recent,
and in fact mysterious.
In this case, the reverse inequality 
$ 
  \lambda_1^\alpha(\Omega) \leq \lambda_1^\alpha(\cB)
$
had been expected (Bareket's conjecture \cite{Bareket_1977} from 1977),
however, its global validity was disproved in \cite{FK7}.
We refer to \cite{Ferone-Nitsch-Trombetti_2015,AFK,Bucur-Ferone-Nitsch-Trombetti}
for latest results in this direction.
The spectral optimisation of the Robin Laplacian including 
Theorem~\ref{Thm.lambda1} is surveyed in the recent monograph
by Bandle and Wagner~\cite{BW}. 

Property~(b) is an exterior variant of a result due to Courant
from 1918 \cite{Courant_1918} for the first Dirichlet eigenvalue,
which actually preceded the Faber--Krahn inequality.
It turns out that the isoperimetric constraint is conceptually
easier than fixing the area of the domain.
As an illustration of this claim, note that 
$ 
  \lambda_1^\alpha(\Omega) \leq \lambda_1^\alpha(\cB)
$
for negative~$\alpha$, where~$\cB$ is the disk of the same perimeter as~$\Omega$,
was established in~\cite{AFK}.
On the other hand, the proof of Bareket's conjecture restated for 
simply connected planar domains still remains open.
It is remarkable that the exterior optimisation is easier 
in the sense that both the isoperimetric and isochoric problems 
for the first eigenvalue 
are fully settled in view of Theorem~\ref{Thm.lambda1}.

\subsection{The second eigenvalue}
The state of the art for the spectral optimisation 
of the first Robin eigenvalue in a bounded domain and its exterior
is our motivation to look at the second eigenvalue 
\begin{equation}\label{Rayleigh2}
  \lambda_2^\alpha(\Omgc)
  :=
  \inf_{
		\begin{smallmatrix}\cL\subset H^1(\Omgc)\\
	\dim\cL = 2\end{smallmatrix}}\sup_{u\in\cL\sm\{0\}}
  \frac{\displaystyle \int_{\Omgc}|\nb u|^2 
  +\alpha\int_{\p\Omg}|u|^2}{\displaystyle\int_{\Omgc}|u|^2}
  \,.
\end{equation}
More precisely, $\lambda_2^\alpha(\Omgc)$ represents 
the second eigenvalue of the Robin Laplacian in $L^2(\Omgc)$
only if~$\alpha$ is negative and~$|\alpha|$ is sufficiently large.
For small values of~$|\alpha|$, $\lambda_2^\alpha(\Omgc)=0$
is the lowest point in the essential spectrum. 
We expect the validity of the following conjecture.
\begin{conj}\label{Conj.lambda2}
If~$\Omega$ is simply connected, then
\begin{enumerate}
\item[\emph{(a)}]
$ 
  \lambda_2^\alpha(\Omgc) \leq \lambda_2^\alpha(\cBc)
$
where~$\cB$ is the disk of the same area as~$\Omega$;
 \item[\emph{(b)}]
$ 
  \lambda_2^\alpha(\Omgc) \leq \lambda_2^\alpha(\cBc)
$
where~$\cB$ is the disk of the same perimeter as~$\Omega$.
\end{enumerate}
\end{conj}

Again the validity of the inequalities is trivial if~$\alpha$ is non-negative
or if it is negative but~$|\alpha|$ is small.  
Determining the critical (negative) value $\alpha_\star(\Omgc)$
for which $\lambda_2^\alpha(\Omgc)$
becomes a discrete eigenvalue emerging from the essential spectrum,
$$
 \alpha_{\star}(\Omgc) :=
  \sup_{\lambda_2^\alpha(\Omgc) < 0} \alpha
  \,,
$$ 
is part of the problem.
 
Property~(a) is an exterior variant of the Szeg\H{o}--Weinberger inequality
stating that 
$ 
  \lambda_2^{\rm N}(\Omega) \leq \lambda_2^{\rm N}(\cB)
$,
where the Neumann eigenvalues correspond to $\alpha=0$.
This maximal property of the disk was conjectured by 
Kornhauser and Stakgold in 1952~\cite{Kornhauser-Stakgold_1952}
and proved independently by Szeg\H{o} in 1954~\cite{Szego}
and Weinberger in 1956~\cite{Weinberger} by different method
(the latter extends to higher dimensions).
Extensions of the Szeg\H{o}--Weinberger inequality to Robin 
eigenvalues in bounded domains has been recently considered
by Freitas and Laugesen \cite{FL20,FL21}.

We have not been able to establish Conjecture~\ref{Conj.lambda2}
under the stated isochoric and isoperimetric constraints.
Instead, we naturally arrive at constraints due to the \emph{elastic energy}
\begin{equation}\label{elastic}
  E(\p\Omg) := \frac12\int_{\partial\Omega} \kp^2
  \,,
\end{equation}
where~$\kappa$ is the curvature of~$\partial\Omega$.
The main result of this paper reads as follows.

\begin{thm}\label{Thm.lambda2}
If~$\Omega$ is convex, then
\begin{enumerate}
\item[\emph{(i)}]
$ 
  \lambda_2^\alpha(\Omgc) \leq \lambda_2^\alpha(\cBc)
$
where~$\cB$ is the disk with $\p\cB$ of the same elastic energy as $\p\Omg$;
\item[\emph{(ii)}]
$
 \alpha_\star(\Omgc) \geq - \frac{1}{\pi} E(\p\Omg)
$ where the equality is attained for $\Omega$ being a disk.
\end{enumerate}
\end{thm}

Below (\cf~Theorem~\ref{thm2}),
we restate the theorem under the additional assumption that~$\Omega$
is not a disk congruent to~$\cB$, which implies the strong inequality
$ 
  \lambda_2^\alpha(\Omgc) < \lambda_2^\alpha(\cBc)
$
for all $\aa < -\frac{1}{R}$. 
In Theorem~\ref{Thm.lambda2},
the statement~(i) is of course trivial if $\alpha \geq \alpha_\star(\Omgc)$
(including the case of non-negative~$\alpha$). 
It remains an open question whether the equality in the inequality in Theorem~\ref{Thm.lambda2}\,(ii)
is attained only for disks.

Despite the fact that Theorem~\ref{Thm.lambda2} 
does not imply Conjecture~\ref{Conj.lambda2},
it can be considered as an analogy of the isoperimetric 
constraint of Courant's preceding the ultimate isochoric result 
of Faber and Krahn in the Dirichlet situation of bounded domains. 
Moreover, $E(\p\Omg)$ is a two-dimensional variant of the Willmore's energy,
which already appeared naturally for the spectral optimisation  
of $\lambda_1^\alpha(\Omgc)$ in three 
dimensions~\cite{KL20}.

Theorem~\ref{Thm.lambda2} substantially improves~\cite{EL22},
where the inequality as in Theorem~\ref{Thm.lambda2}\,(i) was obtained under 
a more restrictive pointwise condition on the curvature of~$\partial\Omega$
(\cf~Remark~\ref{Rem.EL22}).

\subsection{Monotonicity}
Another objective of this paper is to analyse monotonicity 
of the second eigenvalue $\lambda_2^\alpha(\Omgc)$ 
with respect to a domain inclusion.  
More specifically, it is well known that, 
in the Dirichlet case,
$\lambda_k^{\rm D}(\Omega_1) \leq \lambda_k^{\rm D}(\Omega_2)$
for every $k \geq 1$ whenever $\Omega_1 \supset \Omega_2$.
This property is generally false for other boundary conditions
(see, \eg, \cite[Section~1.3.2]{Henrot_2006}). 
However, it can be occasionally achieved under extra hypotheses 
(see~\cite{Pinsky} as an illustration for bounded domains
with combined Dirichlet and Neumann boundary conditions). Recently, the monotonicity of the eigenvalues with respect to inclusion
of convex domains for the Neumann Laplacian was established with a correction factor dependent only on the dimension~\cite{F22, FK23}.

In the present case of attractive Robin boundary conditions ($\alpha < 0$),
the first eigenvalue satisfies
the reverse monotonicity $\lm_1^\alpha(\Omgc) \le \lm_1^\alpha(\cBc)$ 
provided that $\Omg\supset\cB$. 
This is a simple consequence of the isoperimetric inequality 
in Theorem~\ref{Thm.lambda1}\,(i) combined with~\cite[Proposition~3.2]{KL18}.
This result can be viewed as an exterior analogue of~\cite[Theorem~1]{GS05}.
 
In this paper, we address the question whether 
$\lm_2^\alpha(\Omgc) \le \lm_2^\alpha(\cBc)$ for any $\alpha < 0$ 
under the same inclusion assumption $\Omg\supset\cB$. 
We obtain an affirmative answer under additional geometric assumptions on~$\Omg$.
\begin{thm}\label{Thm.mono}
    If~$\Omega$ is strictly star-shaped and centrally symmetric,
	$\cB$ is a disk of radius~$R$ centred at the origin
	and $\Omg\supset\cB$, then
	\begin{enumerate} 
	\item[\emph{(i)}]
	$\lambda_2^\alpha(\Omgc) \leq \lambda_2^\alpha(\cBc)$;
	\item[\emph{(ii)}] 
	$\alpha_\star(\Omgc) \ge -\frac{1}{R}$.
	\end{enumerate}
\end{thm}

Below (\cf~Theorem~\ref{thm1}),
we restate the theorem under the additional assumption $\Omg\not=\cB$,
which implies the strict inequality $\lambda_2^\alpha(\Omgc) < \lambda_2^\alpha(\cBc)$ 
for all $\alpha < -\frac{1}{R}$.
In Theorem~\ref{Thm.mono},
the statement~(i) is of course trivial if $\alpha \geq \alpha_\star(\Omgc)$
(including the case of non-negative~$\alpha$). 
We leave as an open problem 
whether the hypotheses about~$\Omega$ 
in Theorem~\ref{Thm.mono} are necessary 
to have the monotonicity.

We also emphasise that the analogous monotonicity 
of the second eigenvalue for bounded domains
remains unresolved in the full generality. 
However, under certain restrictions on the value of~$\alpha$ in terms of~$R$,
the desired monotonicity for bounded domains 
follows from the isoperimetric inequality proved in~\cite[Theorem~A]{FL21}.

\subsection{Structure of the paper}
In Section~\ref{Sec.pre}, we collect basic results about the Robin Laplacian
in exterior domains. 
In particular, we establish its spectral properties 
in the exterior of disks  
(\cf~Proposition~\ref{prop:disk}).
The monotonicity properties are studied 
and Theorem~\ref{Thm.mono} established in Section~\ref{Sec.mono}.
Finally, in Section~\ref{Sec.optim}, we perform the spectral optimisation
under the isoelastic constraint and prove Theorem~\ref{Thm.lambda2}.

\section{Preliminaries}\label{Sec.pre}
\subsection{The Robin Laplacian on a planar exterior domain}
In this subsection we will introduce the Robin Laplacian on a planar exterior domain and recall its basic spectral properties. Particular attention will be paid to the Robin Laplacian on the exterior of a centrally symmetric bounded domain.
For the Robin problem in bounded domains, 
see \cite[Sec.~4]{Henrot_2017}.

Let $\Omg\subset\dR^2$ be a bounded simply-connected domain with $C^\infty$-smooth boundary $\p\Omg$. We abbreviate by $\nu$ the outer unit normal vector to $\p\Omg$ and by $\der\s$ the one-dimensional Hausdorff measure on $\p\Omg$.
We denote $\Omgc := \dR^2\setminus\ov\Omg$ the open complement of $\Omg$. The exterior domain $\Omgc\subset\dR^2$ is connected and unbounded and has a compact boundary, which coincides with $\p\Omg$. Let the boundary parameter $\alpha\in\dR$ be fixed. The $L^2$-based standard Sobolev space on $\Omgc$ of order $k\in\dN$ will be denoted by $H^k(\Omgc)$.
According to~\cite[Section 2]{KL18} the symmetric, densely defined quadratic form
\begin{equation}\label{eq:form}
	\frh^{\Omgc}_\alpha[u] := \int_{\Omgc}|\nb u|^2\dd x +\alpha\int_{\p\Omg}|u|^2\dd \s,\qquad \dom\frh^{\Omgc}_\alpha := H^1(\Omgc),
\end{equation}
is closed and lower-semibounded in the Hilbert space $L^2(\Omgc)$. By the first representation theorem~\cite[Theorem VI 2.1]{K}, a unique self-adjoint operator $-\Delta^{\Omgc}_\alpha$ in the Hilbert space $L^2(\Omgc)$ is associated to the quadratic form $\frh^{\Omgc}_\alpha$. The operator $-\Delta^{\Omgc}_\alpha$ is called the \emph{Robin Laplacian} on $\Omgc$.
It follows \eg~from~\cite[Proposition 3.1 and Theorem 3.5]{BLLR17} that the Robin Laplacian on the exterior domain $\Omgc$ is characterised by
\begin{equation}\label{key}
	-\Delta_\alpha^{\Omgc} u = -\Delta u,
	\qquad \dom(-\Delta_\alpha^{\Omgc})
	= \big\{u\in H^2(\Omgc)\colon \p_\nu u|_{\p\Omg} = \alpha u|_{\p\Omg}\big\},  
\end{equation}
where $u|_{\p\Omg}\in H^{3/2}(\p\Omg)$ is the trace of $u$ on the boundary and where $\p_\nu u|_{\p\Omg}\in H^{1/2}(\p\Omg)$ is the normal derivative of $u$ on the boundary with the normal pointing outwards of $\Omg$ (inwards of $\Omgc$, respectively).  

In the next proposition we collect basic spectral properties of $-\Delta_\alpha^{\Omgc}$.
\begin{prop}\cite[Propositions 1 and 2]{KL18}
	Let the self-adjoint operator $-\Delta_\alpha^{\Omgc}$ in $L^2(\Omgc)$ be associated with the quadratic form~\eqref{eq:form}. Then the following hold.
	\begin{myenum}
	\item The essential spectrum of $-\Delta_\alpha^{\Omgc}$ coincides with $[0,\infty)$.
	\item The negative discrete spectrum of $-\Delta_\alpha^{\Omgc}$ is non-empty for any $\alpha <0$.
	\end{myenum}
\end{prop} 
In the following, we will always assume that $\alpha <0$. According to~\cite[Theorem 6.9]{B62} the negative discrete spectrum of $-\Delta_\alpha^{\Omgc}$ is finite for any $\alpha < 0$ and we denote by $N_\alpha(\Omgc) \in\dN$ the dimension of the negative spectral subspace of $-\Delta_\alpha^{\Omgc}$.
We denote by $\{\lm_k^\alpha(\Omgc)\}_{k=1}^{N_\alpha(\Omgc)}$ the negative eigenvalues of $-\Delta_\alpha^{\Omgc}$ enumerated in the non-decreasing order and repeated with multiplicities taken into account. For the sake of convenience, we extend the sequence of negative eigenvalues and set $\lm_{k+N_\alpha(\Omgc)}^\alpha(\Omgc) := 0$ for all $k\in\dN$. 
The extended sequence of the eigenvalues of $-\Delta_\alpha^{\Omgc}$ can be characterised by the min-max principle~\cite[\S 4.5]{D95} (see also~\cite[Theorem 1.27 and 1.28]{FLW23})
\begin{equation}\label{eq:minmax}
\begin{aligned}
	\lm_k^\alpha(\Omgc) &\!=\! \inf_{
		\begin{smallmatrix}\cL\subset H^1(\Omgc)\\
	\dim\cL = k\end{smallmatrix}}\sup_{u\in\cL\sm\{0\}}
\frac{\displaystyle \int_{\Omgc}|\nb u|^2\dd x +\alpha\int_{\p\Omg}|u|^2\dd\s}{\displaystyle\int_{\Omgc}|u|^2\dd x}\\
   &\!=\! \sup_{
	u_1,\dots,u_{k-1}\in L^2(\Omgc)}\inf_{\begin{smallmatrix}u\in
	H^1(\Omgc)\cap\{u_1,\dots,u_{k-1}\}^\bot\\
	u\ne 0\end{smallmatrix}}\!\!
\frac{\displaystyle \int_{\Omgc}|\nb u|^2\dd x +\alpha\int_{\p\Omg}|u|^2\dd\s}{\displaystyle\int_{\Omgc}|u|^2\dd x},\qquad k\in\dN;
\end{aligned}
\end{equation}
here, the infimum in the first characterisation
of the extended sequence of eigenvalues of $-\Delta_\aa^{\Omgc}$ is taken over $k$ dimensional linear subspaces of $H^1(\Omgc)$;
the supremum in the second characterisation
in the case that $k\le N_\aa(\Omgc)$ is attained when  $u_1,u_2,\dots,u_{k-1}$ are a family of orthonormal eigenfunctions of $-\Delta_\alpha^{\Omgc}$ corresponding to the eigenvalues $\lm_1^\alpha(\Omgc), \lm_2^\alpha(\Omgc),\dots,\lm_{k-1}^\alpha(\Omgc)$ and the infimum is then attained when $u$ is an eigenfunction of $-\Delta_\aa^{\Omgc}$ corresponding to the eigenvalue $\lm_k^\aa(\Omgc)$.

By the min-max principle~\eqref{eq:minmax} the function $N_\alpha(\Omgc)$ is non-increasing in $\alpha$ on the whole interval $(-\infty,0)$ and it follows from~\cite[Corollary 1.4]{PP16} that
\[
	N_\alpha(\Omgc)\arr \infty,\qquad\text{as}\,\,\alpha\arr-\infty.
\]
\begin{dfn}
The constant $\alpha_\star(\Omgc)\le 0$ such that $N_\alpha(\Omgc)\ge 2$ if, and only if, $\alpha < \alpha_\star(\Omgc)$, is called \emph{the critical coupling constant}.
\end{dfn}
We point out that $\alpha_\star(\Omgc) = 0$ can not occur. In order to see this, let us choose a disk $\cB\subset\dR^2$ so that the inclusion $\ov\Omg\subset\cB$ holds. 
Let us introduce the domain $\Omg' :=\cB\sm\ov\Omg$ and consider the
self-adjoint operator $-\Delta_\aa^{\Omg'}$ in $L^2(\Omg')$ associated with the closed, densely defined, symmetric, and lower-semibounded quadratic form
\[
	H^1(\Omg')\ni u\mapsto \int_{\Omg'}|\nb u|^2\dd x+\aa\int_{\p\Omg}|u|^2\dd \s.
\]
The operator $-\Delta_0^{\Omg'}$ (with $\aa = 0$) is non-negative and its lowest simple eigenvalue is equal to zero. It follows from the trace theorem~\cite[Theorem 3.37]{McL} combined with~\cite[Theorem VI 3.4]{K}
that the operators $-\Delta_\aa^{\Omg'}$ converge in the norm resolvent sense to $-\Delta_0^{\Omg'}$ as $\aa\arr0$. Thus, we get from~\cite[Satz 9.24]{W0} that the dimension of the negative spectral subspace for $-\Delta_\aa^{\Omg'}$ is equal to one for all $\aa < 0$ with  sufficiently small $|\aa|$. 
As it follows from the min-max principle that the dimension of the negative spectral subspace for $-\Delta_\aa^{\Omg'}$ is not smaller than $N_\aa(\Omgc)$, we get that $N_\aa(\Omgc) = 1$ for all $\aa <0$ with sufficiently small absolute value.

By standard methods (see~\cite[Theorem 8.38]{GT}) it can be shown that the lowest eigenvalue of $-\Delta_\alpha^{\Omgc}$ is simple and that the respective eigenfunction $u_{1,\Omgc}^\alpha$ can be selected to be real-valued and positive in $\Omgc$ (see also the discussion in~\cite[Section 2]{KL18}). 

Recall that the domain $\Omg$ is said to be \emph{centrally symmetric} if it coincides with itself upon rotation by the angle equal to $\pi$ with respect to the origin.
Let us also introduce the mapping
\begin{equation}\label{eq:J}
	\sfJ\colon\dR^2\arr\dR^2,\qquad\sfJ x :=-x.
\end{equation}  
The central symmetry of the domain $\Omg$ is equivalent to the invariance $\sfJ(\Omg) = \Omg$.
Next, we will state a proposition on the
eigenfunction  corresponding to the lowest eigenvalue of $-\Delta^{\Omgc}_\alpha$ for centrally symmetric $\Omg$. 
\begin{prop}\label{prop:symm}
	Assume that the simply-connected  $C^\infty$-smooth domain $\Omg\subset\dR^2$
	is centrally symmetric. Let the mapping $\sfJ$ be as in~\eqref{eq:J}.
	Let $\alpha < 0$ be fixed and the self-adjoint operator $-\Delta_\alpha^{\Omgc}$ in $L^2(\Omgc)$ be associated with the quadratic form~\eqref{eq:form}. Then the eigenfunction $u_{1,\Omgc}^\alpha$ corresponding to the lowest eigenvalue of $-\Delta_\aa^{\Omgc}$ 
	satisfies 
	\begin{equation}\label{eq:gs_symmetry}
		u_{1,\Omgc}^\alpha(x) = u_{1,\Omgc}^\alpha(\sfJ x),\qquad \text{for all}\,\,\, x\in\Omgc.
	\end{equation}
\end{prop}
\begin{proof}
	We introduce the function $v\colon\Omgc\arr\dR$ by
	\[
		v(x) := u_{1,\Omgc}^\aa(\sfJ x),\qquad x\in\Omgc.
	\]
	In view of central symmetry of $\Omg$ the function $v$ is well defined, because for any $x\in\Omgc$ one has $\sfJ x\in\Omgc$.
	It is straightforward to see 
	\[	
\begin{aligned}
		\int_{\Omgc}|\nb v|^2\dd x =& \int_{\Omgc}|\nb u_{1,\Omgc}^\aa|^2\dd x,\qquad
		\int_{\p\Omg}|v|^2\dd \s = \int_{\p\Omg}| u_{1,\Omgc}^\aa|^2\dd \s,\\
		&\int_{\Omgc}| v|^2\dd x = \int_{\Omgc}| u_{1,\Omgc}^\aa|^2\dd x.
\end{aligned}
	\]
	Hence, we infer that $v\in H^1(\Omgc)$ and arrive at
	\[
		\lm_1^\aa(\Omgc) = \frac{\displaystyle\int_{\Omgc}|\nb v|^2\dd x + \aa\int_{\p\Omg}|v|^2\dd \s}{\displaystyle\int_{\Omgc} |v|^2\dd x}.
	\]
	Thus, it follows from~\cite[\S 10.2, Theorem 1]{BS} that $v$ is an eigenfunction of $-\Delta_\aa^{\Omgc}$ corresponding to its lowest eigenvalue $\lm_1^\aa(\Omgc)$. Simplicity of the lowest eigenvalue of $-\Delta_\aa^{\Omgc}$ yields that either $v = u_{1,\Omgc}^\aa$ or $v = -u_{1,\Omgc}^{\aa}$. The second alternative can not occur, because the ground state $u_{1,\Omgc}^\aa$ is positive. Thus, the property~\eqref{eq:gs_symmetry} follows.
\end{proof}

\subsection{The Robin Laplacian on the exterior of a disk}
In this subsection we will analyse the negative discrete spectrum of the Robin Laplacian on the exterior of a disk with an emphasis on the lowest two negative eigenvalues.

Let $\cB\subset\dR^2$ be the disk of radius $R > 0$ centred at the origin. We introduce standard polar coordinates $(r,\tt)$ on $\cBc$.
Let us also introduce the complete family of mutually orthogonal projections
\[	
	\Pi_n \colon L^2(\cBc)\arr L^2(\cBc),\qquad (\Pi_n u)(r,\tt) :=
	\frac{e^{\ii n\tt}}{2\pi}\int_0^{2\pi}
	u(r,\tt') e^{-\ii n\tt'}\dd\tt',\qquad n\in\dZ,
\]
and the unitary mappings
\[
	\sfU_n\colon \ran\Pi_n\arr L^2((R,\infty);r\der r),\quad (\sfU_n u)(r) := \frac{1}{\sqrt{2\pi}}\int_0^{2\pi}
		u(r,\tt')e^{-\ii n\tt'}\dd \tt',\qquad n\in\dZ.
\] 
The family of projections $\{\Pi_n\}_{n\in\dZ}$ induces an orthogonal decomposition 
\[
	L^2(\cBc) = \bigoplus_{n\in\dZ} \ran\Pi_n\simeq \bigoplus_{n\in\dZ} L^2((R,\infty);r\der r).
\]
It is easy to check using~\cite[Propositon 1.15]{S} that $\ran\Pi_n$, $n\in\dZ$, is a reducing subspace for $-\Delta_\alpha^{\cBc}$. Hence, the operator
$-\Delta_\alpha^{\cBc}$ can be decomposed into the orthogonal sum
\begin{equation}\label{eq:ortho}
	-\Delta_\alpha^{\cBc} = \bigoplus_{n\in\dZ}
	\big(\sfU_n^{-1}\sfH_{\alpha,R,n}\sfU_n\big),
\end{equation}
where the self-adjoint fibre operators $\sfH_{\alpha,R,n}$ in $L^2((R,\infty);r\der r)$ are associated with the closed, symmetric, densely defined, and semi-bounded 
quadratic forms
\begin{equation}\label{eq:fiber_form}
\begin{aligned}
	\frh_{\alpha,R,n}[f] &:= \frh_{\alpha}^{\cBc}[\sfU_n^{-1}f] = \int_R^\infty\left(|f'(r)|^2 + \frac{n^2}{r^2}|f(r)|^2\right)r\dd r+\alpha R|f(R)|^2,\\ 
	\dom\frh_{\alpha,R,n} &:= \big\{f\in L^2((R,\infty);r\der r)\colon \sfU_n^{-1} f\in H^1(\cBc)\big\}
	=
	\big\{f\colon f,f'\in L^2((R,\infty);r\der r)\big\}. 
\end{aligned}
\end{equation}
Using the integration by parts we find that the fibre operator is characterised by
\begin{equation}\label{eq:fiber_operator}
\begin{aligned}
	\sfH_{\alpha,R,n}f &= -f''(r)-\frac{f'(r)}{r} + \frac{n^2f(r)}{r^2},\\
	\dom\sfH_{\alpha,R,n} &=\left\{f\colon
	f,f''+\frac{f'}{r}\in L^2((R,\infty);r\der r), f'(R) = \alpha f(R)\right\}.
\end{aligned}
\end{equation}
In view of~\cite[Satz 13.19]{W}, the fibre operator $\sfH_{\alpha,R,n}$ is a rank-one perturbation in the sense of the resolvent difference of a non-negative self-adjoint operator $\sfH_{0,R,n}$ (with $\alpha = 0$).
Hence, by~\cite[\S 9.3, Theorem 3]{BS} the dimension of
the negative spectral subspace of each fibre operator is at most one. It is also straightforward to observe by constructing singular sequences and using compact perturbation argument that $\sess(\sfH_{\alpha,R,n}) = [0,\infty)$ for all $n\in\dZ$.  

Let us introduce the notation for  the lowest spectral point of $\sfH_{\alpha,R,n}$
\[
\lm_1^\alpha(R,n):=\inf\s(\sfH_{\aa,R,n})\le0.
\]
Notice also that $\sfH_{\alpha,R,n} = \sfH_{\alpha,R,-n}$ for any $n\in\dZ$.
Since the domain of $\frh_{\aa,R,n}$ is independent of $n\in\dZ$ and since  for any non-trivial $f\in \dom\frh_{\aa,R,n}$ there holds $\frh_{\alpha,R,m}[f] > \frh_{\alpha,R,n}[f]$ for $|m| > |n|$, we infer that for $|m| > |n|$ by the min-max principle
\[
	\lm^\aa_1(R,m) \ge \lm_1^\aa(R,n),
\]
the inequality being strict provided that 
$\lm_1^\aa(R,m) < 0$.
Thus, it follows from finiteness of the negative discrete spectrum of $-\Delta_\alpha^{\cBc}$ and the orthogonal decomposition~\eqref{eq:ortho}, that there exists $n_\star = n_\star(\alpha)\in\dN_0$ such that
$\lm_1^\aa(R,n) < 0$ if, and only if, $|n| \le n_\star(\alpha)$
and that the negative eigenvalues of $-\Delta_\aa^{\cBc}$ repeated with multiplicities taken into account are given by the sequence 
\[
\lm_1^\aa(R,0) < \lm_1^\aa(R,1)\le
\lm_1^\aa(R,1)<\lm_1^\aa(R,2)\le
\lm_1^\aa(R,2)<\dots<\lm_1^\aa(R,n_\star)\le \lm_1^\aa(R,n_\star).
\]	
The total dimension of the negative spectral subspace of $-\Delta_\alpha^{\cBc}$ is thus given by $N_\alpha(\cBc) = 2n_\star(\alpha)+1$.
The negative discrete spectrum of $-\Delta_\aa^{\cBc}$ consists of the lowest simple eigenvalue corresponding to the fibre $\sfH_{\aa,R,0}$ and of at most finitely many higher negative eigenvalues
of multiplicity two each, corresponding to the fibres $\sfH_{\aa,R,\pm n}$ with $n = 1,2,\dots,n_\star(\aa)$.

In the next proposition we characterise the lowest and the second eigenvalues of $-\Delta_\alpha^{\cBc}$. In the following
$K_\nu$ stands for the modified Bessel function of the second kind and order $\nu\in\dR$.
\begin{prop}\label{prop:disk}
	Let $\cB$ be the disk of radius $R>0$ centred at the origin.
	Let $\alpha < 0$ and the self-adjoint operator $-\Delta_\alpha^{\cBc}$ in $L^2(\cBc)$ be associated with the quadratic form as in~\eqref{eq:form}. Then the following hold.
	\begin{myenum}
	\item 
	The lowest eigenvalue of $-\Delta_{\aa}^{\cBc}$ is given by $\lm_1^\aa(\cBc) = -\xi^2$, where
	$\xi > 0$ is the unique solution of the transcendental equation
	\begin{equation*}
	-\xi K_0'(\xi R) + \aa K_0(\xi R) = 0.
	\end{equation*}
	The  eigenfunction of $-\Delta_\alpha^{\cBc}$ corresponding to its lowest eigenvalue 
	reads in polar coordinates as
	\begin{equation*}\label{eq:1stef}
		u_{1,\cBc}^{\alpha}(r,\tt) = K_0(\xi r).
	\end{equation*}
	Moreover, the function $(0,\infty)\ni R\mapsto \lm_1^\aa(\cBc)$ is strictly decreasing.
	\item 
	The critical coupling constant for $\cBc$ is given by $\alpha_\star(\cBc) = -\frac{1}{R}$. For $\alpha <-\frac{1}{R}$,
	the second eigenvalue of $-\Delta_{\aa}^{\cBc}$ is given by $\lm_2^\aa(\cBc) = -\omg^2$, where
	$\omg > 0$ is the unique solution of the transcendental equation
	\begin{equation}\label{eq:transc}
		-\omg K_1'(\omg R) + \aa K_1(\omg R) = 0.
	\end{equation}
	The orthogonal eigenfunctions of $-\Delta_\alpha^{\cBc}$ corresponding to its second eigenvalue 
	read in polar coordinates as
	\begin{equation}\label{eq:2ndefs}
		\begin{cases}
		u_{2,\cBc}^{\alpha}(r,\tt) = K_1(\omg r) \cos\tt,\\[0.5ex]
		v_{2,\cBc}^{\alpha}(r,\tt) = K_1(\omg r) \sin\tt.
		\end{cases}
	\end{equation}
	Moreover, the function $(-\frac{1}{\aa},\infty)\ni R\mapsto \lm_2^{\aa}(\cBc)$ is strictly decreasing. 
	\end{myenum}
\end{prop}
\begin{proof}
	(i) All the statements of this item are shown in \cite[Section 3]{KL18}.

	\smallskip
	
	\noindent (ii)
	It follows from the analysis preceding this proposition that the dimension of the negative spectral subspace of $-\Delta_\aa^{\cBc}$ is larger than one if, and only if, the fibre operator $\sfH_{\aa,R,1}$ ($n=1$) has a negative eigenvalue. In this case the second eigenvalue of $-\Delta_\aa^{\cBc}$ coincides with the negative eigenvalue $\lm_1^\aa(R,1) < 0$ of $\sfH_{\aa,R,1}$.
	In view of~\eqref{eq:fiber_operator} the fibre operator $\sfH_{\alpha,R,1}$  has a negative eigenvalue in the case that the following system
	\begin{equation}\label{eq:system}	
		\begin{cases}
		-f''(r) -\frac{f'(r)}{r} + \frac{f(r)}{r^2} = \lm f(r),&\qquad\text{on}\,\,(R,\infty),\\
		f'(R) = \alpha f(R),&
		\end{cases}
	\end{equation}
	has a non-trivial solution in $L^2((R,\infty);r\der r)$ for some value $\lm < 0$. The second eigenvalue $\lm_2^{\alpha}(\cBc)$ of $-\Delta_\alpha^{\cBc}$ is then given by the value of $\lm < 0$ for which
	the system~\eqref{eq:system} 	has a non-trivial solution in $L^2((R,\infty);r\der r)$.
	 By~\cite[10.25.1-3]{OLBC} the solution in $L^2((R,\infty);r\der r)$ of the ordinary differential equation in the above system is given by 
	\begin{equation}\label{eq:f}
		f(r) = K_1(\omg r)
	\end{equation}
	(up to a multiplication
	by a constant factor), where we use the abbreviation $\omg = \sqrt{-\lm}$. This solution fulfils the boundary conditions if the following transcendental equation is satisfied
	\[
		-\omg K_0(\omg R)-
		\frac{K_1(\omg R)}{R} = \alpha K_1(\omg R),
	\]
	where we employed that $K_1'(x) =-K_0(x)-\frac{K_1(x)}{x}$ (see~\cite[10.6.2]{OLBC}).
	The existence of the negative eigenvalue for $\sfH_{\alpha,R,1}$ is equivalent to
	existence of a solution $\omg > 0$ to the above transcendental equation.
	Let us  rewrite the equation as
	\begin{equation}\label{eq:transc2}
		-\frac{\omg R K_0(\omg R)}{K_1(\omg R)}  = \alpha R + 1.
	\end{equation}
	The function
	\[
		F(x) := \frac{x K_0(x)}{K_1(x)},\qquad x > 0,
	\]
	is clearly continuous and positive. Moreover, it follows from the inequalities (see~\cite[Theorem 1]{S11})
	\[
	\frac{2x}{1+\sqrt{1+4x^2}}\le\frac{K_0(x)}{K_1(x)} < 1,\qquad x > 0,
	\]
	that 
	\[
		\lim_{x\arr0^+} F(x) = 0\qquad\text{and}\qquad\lim_{x\arr\infty} F(x) = \infty.
	\]
	Therefore, the equation~\eqref{eq:transc2} has a solution $\omg> 0$ if, and only if, $\alpha R + 1 < 0$. Hence, we conclude that
	\[
		\alpha_\star(\cBc) = -\frac{1}{R},
	\]
	Under the assumption $\alpha < -\frac{1}{R}$ clearly both operators
	$\sfH_{\alpha,R,-1}$ and $\sfH_{\alpha,R,1}$ have the same negative eigenvalue.
	It follows from the orthogonal decomposition~\eqref{eq:ortho} and 
	the form~\eqref{eq:f} of the solution  in $L^2((R,\infty);r\der r)$ of the system~\eqref{eq:system} that the orthogonal eigenfunction of $-\Delta_\alpha^{\cBc}$ corresponding to the doubly degenerate eigenvalue $\lm_2^\alpha(\cBc) < 0$ can be expressed in polar coordinates as
	\[
		u_{2,\cBc}^{\alpha,\pm}(r,\tt) := K_1(\omg r)e^{\pm\ii \tt}, 
	\]
	where $\omg > 0$ is the unique solution of the transcendental equation~\eqref{eq:transc}. Moreover, the second eigenvalue $-\Delta_\aa^{\cBc}$ is expressed as $\lm_2^\aa(\cBc) = -\omg^2$.
	 The representations of the orthogonal eigenfunctions of $-\Delta_\alpha^{\cBc}$ corresponding to $\lm_2^\alpha(\cBc)$ stated in~\eqref{eq:2ndefs} follow by taking linear combinations 
	  \[
		u_{2,\cBc}^\alpha = \frac12\left(u_{2,\cBc}^{\alpha,+}+
		u_{2,\cBc}^{\alpha,-}\right)\qquad\text{and}\qquad
		v_{2,\cBc}^\alpha = \frac{1}{2\ii}
		\left(u_{2,\cBc}^{\alpha,+}-u_{2,\cBc}^{\alpha,-}\right).
	\] 

	In order to show strict decay of the second eigenvalue as a function of radius on the interval $(-\frac{1}{\aa},\infty)$ it suffices to show that the negative eigenvalue of $\sfH_{\aa,R,1}$ is a strictly decreasing function of the radius on the same interval. Let $R_2 > R_1 > -\frac{1}{\aa}$ be fixed. Let also $f_1\in\dom\sfH_{\aa,R_1,1}\subset\dom\frh_{\aa,R_1,1}$ be the eigenfunction of $\sfH_{\aa,R_1,1}$ corresponding to its lowest eigenvalue. 
	We define the function $f_2\in\dom\frh_{\aa,R_2,1}$ by
	\[
		f_2(r) := f_1\left(r + R_1-R_2\right).
	\]	
	Clearly, one has $f_2\in\dom\frh_{\aa,R_2,1}$.
	Using $f_2$ as a trial function for $\sfH_{\aa,R_2,1}$ we get 
	\[
	\begin{aligned}
		\lm_1^\aa(R_2,1) &\le 
		\frac{\displaystyle\int_{R_2}^\infty
		\left(\left|f_1'(r+R_1-R_2)\right|^2 + \frac{1}{r^2}\left|f_1(r+R_1-R_2)\right|^2\right)r\dd r+\aa R_2|f_1(R_1)|^2}{\displaystyle\int_{R_2}^\infty			\left|f_1\left(r+R_1-R_2\right)\right|^2r
		\dd r}\\
		&=
		\frac{\displaystyle\int_{R_1}^\infty
			\left(\left|f_1'(r)\right|^2 + \frac{1}{(r+R_2-R_1)^2}\left|f_1(r)\right|^2\right)
			\frac{R_1}{R_2}\left(r+R_2-R_1\right)\dd r+\aa R_1|f_1(R_1)|^2}{\displaystyle\int_{R_1}^\infty			\left|f_1\left(r\right)\right|^2
			\frac{R_1}{R_2}\left(r+R_2-
			R_1\right)
			\dd r}\\
		& <
		\frac{\displaystyle\int_{R_1}^\infty
			\left(\left|f_1'(r)\right|^2 + \frac{1}{r^2}\left|f_1(r)\right|^2\right)
			r\dd r+\aa R_1|f_1(R_1)|^2}{\displaystyle\int_{R_1}^\infty			\left|f_1\left(r\right)\right|^2
			r
			\dd r} = \lm_1^\aa(R_1,1),	\end{aligned}
	\]	
	where we used the min-max principle in the first step, performed the change of variables $r\mapsto r+R_1-R_2$ and multiplied the numerator and the denominator by $\frac{R_1}{R_2}$ in the second step, 
	combined the inequalities $\frac{R_1}{R_2}(r+R_2-R_1)< r$ (for all $r > R_1$) and $\frh_{\aa,R_1,1}[f_1] < 0$ in the third step, and used that $f_1$ is an eigenfunction of $\sfH_{\aa,R_1,1}$ corresponding to its lowest eigenvalue in the last step. 
\end{proof}

\section{Monotonicity of the second Robin eigenvalue on an exterior domain}\label{Sec.mono}
In this section, we establish a stronger version of Theorem~\ref{Thm.mono}.

Let $\Omega\subset\dR^2$ be any planar strictly star-shaped $C^\infty$-smooth bounded domain.  
More specifically, we assume that the boundary~$\partial\Omega$
can be parametrised by the mapping
\begin{equation}\label{curve}
\Gamma: [0,2\pi] \to \dR^2\colon
\big\{
\theta \mapsto \rho(\theta) (\cos\theta,\sin\theta)^\top
\big\}
\,,
\end{equation}
where $\rho \in C^\infty([0,2\pi])$ is a positive function
such that $\rho^{(k)}(0)=\rho^{(k)}(2\pi)$ for $k\in\dN_0$.
For the sake of convenience we can extend the function $\rho$ by periodicity to the whole real line. With a slight abuse of notation, we will denote this extension again by $\rho$.
In the case that $\Omg$ is centrally symmetric, the function $\rho$ satisfies $\rho(\tt+\pi) =\rho(\tt)$ for any $\tt\in\dR$.
In the first main result of the paper we obtain 
a monotonicity property for the second eigenvalue in the class of strictly star-shaped centrally symmetric domains. As a by-product of this construction, we also get an estimate on the critical coupling constant for such domains.
\begin{thm}\label{thm1}
	Let $\Omega\subset\dR^2$ be a bounded $C^\infty$-smooth domain. Assume that
	$\Omg$ is strictly star-shaped and centrally symmetric (with respect to the origin). Assume also that the strict inclusion $\cB\subsetneq\Omg$ holds for the disk $\cB\subset\dR^2$ of radius $R > 0$ centred at the origin.  Then the following hold.
	\begin{myenum} 
	\item
	$\lambda_2^\alpha(\Omgc) < \lambda_2^\alpha(\cBc)$ for all $\alpha < -\frac{1}{R}$.
	\item $\alpha_\star(\Omgc) \ge -\frac{1}{R}$.
	\end{myenum}
\end{thm}
\begin{proof}
We divide this proof in two steps for the sake of convenience of the reader. Both items of the theorem will be proved simultaneously.
	
\noindent\emph{Step 1: orthogonality.}
Let $\alpha <-\frac{1}{R}$ be fixed. Then
by Proposition~\ref{prop:disk} the dimension of the negative spectral subspace of $-\Delta_\alpha^{\cBc}$ is at least three: the lowest eigenvalue of the operator $-\Delta_\alpha^{\cBc}$  is simple, while its second eigenvalue has multiplicity two.  The orthogonal eigenfunctions corresponding to the second eigenvalue of $-\Delta_\alpha^{\cBc}$ given in Proposition~\ref{prop:disk}\,(ii) are represented by
\begin{equation}\label{eq:efs}
	u_{2,\cBc}^\alpha(r,\tt) = K_1(\omg r)\cos\tt\qquad\text{and}\qquad
	v_{2,\cBc}^\alpha(r,\tt) = K_1(\omg r)\sin\tt
\end{equation} 	
with $\lm_2^\alpha(\cBc) =-\omg^2$, where $\omg > 0$ is the unique solution of the transcendental equation~\eqref{eq:transc}.

Recall that by Proposition~\ref{prop:symm}
the ground state $u_{1,\Omgc}^\alpha$ of $-\Delta^{\Omg^{\rm c}}_\alpha$ is centrally symmetric. In polar coordinates, this property of the ground state can be expressed as
\begin{equation}\label{eq:symmetry}
	u_{1,\Omgc}^\alpha(r,\tt) = u_{1,\Omgc}^\alpha(r,\tt+\pi),\qquad\text{for all}\,\, \tt\in[0,\pi),\,\,r > \rho(\tt).
\end{equation}
For the sake of convenience,
by periodicity in the angular variable,
we define $u_{1,\Omgc}^\alpha(r,\tt)$ for any $\tt\in\dR$
provided that $r > \rho(\tt)$.

We remark that the exterior domains satisfy the opposite inclusion $\Omgc\subsetneq\cBc$.
Using the central symmetry of $\Omg$ and~\eqref{eq:symmetry} we obtain 
\begin{equation}\label{eq:orth1}
\begin{aligned}
\int_{\Omg^{\rm c}}u_{2,\cBc}^\alpha(x)u_{1,\Omgc}^\alpha(x)\dd x &= 
\int_{0}^{2\pi}\cos\tt\int_{\rho(\tt)}^\infty
K_1(\omg r)u_{1,\Omgc}^\alpha(r,\tt) \, r\dd r\dd\tt\\
&=
\int_{\pi}^{3\pi}\cos(\hat\tt-\pi)
\int_{\rho(\hat\tt-\pi)}^\infty
K_1(\omg r)u_{1,\Omgc}^\alpha(r,\hat\tt-\pi) \, r\dd r\dd\hat\tt \\
&=
-\int_{0}^{2\pi}\cos\hat\tt	\int_{\rho(\hat\tt)}^\infty
K_1(\omg r)u_{1,\Omgc}^\alpha(r,\hat\tt) \, r\dd r\dd\hat\tt\\
& = -
\int_{\Omg^{\rm c}}u_{2,\cBc}^\alpha(x)u_{1,\Omgc}^\alpha(x)\dd x =0,
\end{aligned}
\end{equation} 
where we performed the substitution $\hat\tt = \pi +\tt$.
Analogously we find that
\begin{equation}\label{eq:orth2}
\begin{aligned}
\int_{\Omg^{\rm c}}v_{2,\cBc}^\alpha(x)u_{1,\Omgc}^\alpha(x)\dd x &= 
\int_{0}^{2\pi}\sin\tt\int_{\rho(\tt)}^\infty
K_1(\omg r)u_{1,\Omgc}^\alpha(r,\tt) \, r\dd r\dd\tt\\
&=
\int_{\pi}^{3\pi}\sin(\hat\tt-\pi)
\int_{\rho(\hat\tt-\pi)}^\infty
K_1(\omg r)u_{1,\Omgc}^\aa(\rho,\hat\tt-\pi) \, r\dd r\dd\hat\tt \\
&=
-\int_{0}^{2\pi}\sin\hat\tt	\int_{\rho(\hat\tt)}^\infty
K_1(\omg r)u_{1,\Omgc}^{\alpha}(r,\hat\tt) \, r\dd r\dd\hat\tt\\
& = -
\int_{\Omg^{\rm c}}v_{2,\cBc}^\alpha(x)u_{1,\Omgc}^\alpha(x)\dd x =0.
\end{aligned}
\end{equation}

\noindent\emph{Step 2: application of the min-max principle.}
Let us introduce the quantity
\begin{equation}	
	\gg_\Omg := \int_0^{2\pi}\frac{\dot{\rho}(\tt)}{\rho(\tt)}K_1(\omg \rho(\tt))^2\sin\tt\cos\tt\dd\tt\in\dR.
\end{equation}
As a trial function for the min-max principle
applied to $-\Delta_\aa^{\Omgc}$ we select the following real-valued function on $\Omgc$
\begin{equation}\label{eq:ustar}
	u_\star := \begin{cases}
	u_{2,\cBc}^\alpha|_{\Omgc},&\quad \gg_\Omg \ge 0,\\[0.3ex]
	v_{2,\cBc}^\alpha|_{\Omgc},&\quad \gg_\Omg <0.
	\end{cases}
\end{equation}
In other words, we choose one of the two trial functions depending on the sign of $\gg_\Omg$.

It is not difficult to see that $u_\star\in H^2(\Omgc)$ being the restriction to $\Omgc$ of a function in the Sobolev space $H^2(\cBc)$. In view of~\eqref{eq:orth1} and~\eqref{eq:orth2}, the function $u_\star$ is orthogonal to the ground state $u_{1,\Omgc}^\alpha$ of $-\Delta_\alpha^{\Omgc}$ in $L^2(\Omgc)$. Finally, by its construction
the function $u_\star$ satisfies the differential equation $-\Delta u_\star = \lm_2^\alpha(\cBc)u_\star$ on $\Omgc$.
Using the second variational characterisation of $\lm_2^\alpha(\Omgc)$ in~\eqref{eq:minmax}, 
we get integrating by parts
\begin{equation}\label{eq:bound}
\begin{aligned}
\lambda_2^\alpha(\Omgc) &\leq \frac{\displaystyle\int_{\Omgc}|\nb u_\star|^2\dd x + \alpha\int_{\p\Omg}|u_\star|^2\dd\s}{\displaystyle\int_{\Omgc}|u_\star|^2\dd x} \\& = 
\frac{\displaystyle\int_{\Omgc} (-\Delta u_\star)u_\star\dd x + \int_{\p\Omg}u_\star\left(-\frac{\p u_\star}{\p\nu} +\alpha u_\star\right)\dd\s}{\displaystyle\int_{\Omgc}|u_\star|^2\dd x}\\
&=
\lm_2^\alpha(\cBc)
+ \frac{\displaystyle\int_{\p\Omg}u_\star\left(-\frac{\p u_\star}{\p\nu} +\alpha u_\star\right)\dd\s}{\displaystyle\int_{\Omgc}|u_\star|^2\dd x},
\end{aligned}
\end{equation}
where $\frac{\p u_\star}{\p\nu}$ stands for the normal derivative of $u_\star$ on $\p\Omg$ with the normal pointing outwards of $\Omg$. In order to prove the inequality in item (i) it remains to show that
\begin{equation}\label{eq:bt_neg}
	\int_{\p\Omg}u_\star\left(-\frac{\p u_\star}{\p\nu} +\alpha u_\star\right)\dd\s <0.
\end{equation}
By using the parameterisation~\eqref{curve}, 
we find $|\dot\Gamma|^2 = \rho^2 + \dot{\rho}^2$ and
\begin{equation}
\nu(\theta) = \frac{1}{|\dot\Gamma(\theta)|}
\left(
\rho(\theta) \cos\theta + \dot{\rho}(\theta) \sin\theta,
\rho(\theta) \sin\theta - \dot{\rho}(\theta) \cos\theta
\right)^\top,
\end{equation}
where $\nu(\tt)$ is the outer unit normal to $\Omg$ at the point $\G(\tt)$.
For the sake of convenience, let us introduce the shorthand notation
\[	
c(\tt) := \begin{cases}
\cos\tt,& \gg_\Omg \ge 0,\\
\sin\tt,& \gg_\Omg < 0,
\end{cases}\qquad\text{and}\qquad J(r) := K_1(\omg r).
\]
By passing to the polar coordinates,
\begin{equation}
\begin{aligned}
\frac{\partial u_\star}{\partial \nu}
&= \nu \cdot \nabla u_\star
= \frac{1}{|\dot\Gamma(\theta)|}
\left(
\rho(\theta) \frac{\partial u_\star}{\partial r}
- \frac{\dot{\rho}(\theta)}{\rho(\tt)} \frac{\partial u_\star}{\partial \theta}  
\right)\\
&= 
\frac{1}{|\dot{\G}(\tt)|}\left(\rho(\tt)  J'(\rho(\tt))c(\tt) - \frac{\dot\rho(\tt)}{\rho(\tt)}
J(\rho(\tt))
c'(\tt)\right)   
.
\end{aligned}\end{equation}
Consequently, we get
\begin{equation*}
\begin{aligned}
-\int_{\p\Omg}
u_\star\frac{\p u_\star}{\partial \nu}\dd \s
&= \int_0^{2\pi}  \left[
-\rho(\theta)J( \rho(\tt)) J'(\rho(\tt)) c^2(\tt)
+ \frac{\dot{\rho}(\theta)}{\rho(\theta)} 
J(\rho(\theta))^2 c'(\tt)c(\tt)
\right]
\der\theta
\\
&=
-\int_0^{2\pi} \rho(\theta) J( \rho(\tt)) J'(\rho(\tt)) c^2(\tt)
\dd\theta - |\gamma_\Omg|\\
&\leq -\int_0^{2\pi} 
\rho(\theta) J(\rho(\tt)) J'(\rho(\theta)) c^2(\theta)\dd\theta.
\end{aligned}  
\end{equation*}
Hence, we get the following estimate for the boundary term appearing in~\eqref{eq:bound}
\begin{equation}\label{eq:estimate}
	\begin{aligned}
	&\int_{\partial\Omg} u_\star 
	\left(-\frac{\partial u_\star}{\partial \nu} + \alpha u_\star \right)\dd\s\\
	&\qquad\leq \int_0^{2\pi} 
	J(\rho(\theta)) \left[
	-\rho(\theta) J'(\rho(\theta)) 
	+ \alpha \sqrt{\rho(\theta)^2+\dot{\rho}(\theta)^2} J(\rho(\theta)) 
	\right] 
	c^2(\theta) \, \der\theta
	\\
	&\qquad\leq \int_0^{2\pi} 
	J(\rho(\theta)) \, \rho(\theta) \left[
	- J'(\rho(\theta)) 
	+ \alpha J(\rho(\theta)) 
	\right] 
	c^2(\theta) \, \der\theta
	\,.
	\end{aligned}  
\end{equation}
	Consider now the following auxiliary
	continuous function
	\begin{equation}
	g(r) := - J'(r) + \alpha J(r) 
	= - \omg K_1'(\omg r) + \alpha K_1(\omg r),\qquad r \ge R.
	\end{equation}
	Observe first that $g(R) = 0$ due to~\eqref{eq:transc}. Using the asymptotic expansion following from~\cite[9.7.2, 9.7.4]{AS} 
	\[
	-\frac{K_1'(x)}{K_1(x)} = 1+\frac{1}{2x}+\cO\left(\frac{1}{x^2}\right),\qquad x\arr\infty,
	\] 
	we get that
	\begin{equation}\label{eq:asymptotics_f}
	g(r) = K_1(\omg r) \left[\omg + \frac{1}{2r}+\aa+\cO\left(\frac{1}{r^2}\right)\right],\qquad r\arr\infty.
	\end{equation}
	It follows from~\cite[Theorem 3.1]{P16} that $\lm_1^\aa(\cBc) \ge -\aa^2$. Using simplicity of the ground-state eigenvalue for $-\Delta_\aa^{\cBc}$ we infer that $\lm_2^\aa(\cBc) > -\aa^2$ or, equivalently, that $\omg +\aa < 0$. Hence, as a consequence of the asymptotics~\eqref{eq:asymptotics_f}
	we get
	that $g(r) < 0$ for all sufficiently large $r > R$. Thus, we conclude from continuity of $g$ that either $g(r) < 0$ for all $r > R$ or there exists $R_1 > R$ such that $g(R_1) = 0$.
	The latter possibility can not occur, because it would mean that $\lm_2^\aa(\cBc) < 0$ is simultaneously the second eigenvalue of the Robin Laplacian on the exterior of the disk of larger radius $R_1 > 0$ with the same boundary parameter $\aa < 0$, which contradicts the monotonicity with respect to radius obtained in Proposition~\ref{prop:disk}\,(ii). Negativity of $g$ combined with the estimate~\eqref{eq:estimate} leads to
	the desired inequality  in~\eqref{eq:bt_neg},
	by which the proof of (i) is complete.
	
	The estimate on the critical coupling constant
	$\aa_\star(\Omgc)$ in (ii)  follows as a by-product.
	Indeed, the eigenvalue inequality in~(i)
	implies that the negative spectral subspace of the operator $-\Delta_\aa^{\Omgc}$
	is of dimension at least two for all $\aa < -\frac{1}{R}$. Hence, we conclude that $\aa_\star(\Omgc)\ge-\frac{1}{R}$.
\end{proof}
\begin{cor}
	Let $\Omega\subset\dR^2$ be a bounded convex $C^\infty$-smooth  domain. Assume that $\Omg$ is centrally symmetric. Then the critical coupling constant satisfies 
	\[
		\alpha_\star(\Omgc) \ge -\frac{1}{r_{\rm i}},
	\]	
	where $r_{\rm i} > 0$ is the in-radius of $\Omg$.
\end{cor}
\begin{proof}
	First, we point out that the domain $\Omg$ is strictly star-shaped with respect to the origin.
	Thus, in view of Theorem~\ref{thm1}\,(ii) it suffices to show that the origin can be chosen as the centre of an inscribed disk for $\Omg$ of the largest radius. Suppose that $\cB_{r_{\rm i}}(x_0)\subset\Omg$ 
	for some $x_0\in\Omg$, which is not necessarily the origin. By the symmetry we can inscribe in $\Omg$ the disk $\cB_{r_{\rm i}}(-x_0)$ of the same radius centred at $-x_0$. By convexity  we get that $\Omg$
	contains the convex hull of $\cB_{r_{\rm i}}(x_0)\cup\cB_{r_{\rm i}}(-x_0)$. For simple geometric reasons this convex hull contains the disk $\cB_{r_{\rm i}}(0)$ (centred at the origin). Thus, the origin can also be chosen as a centre of an inscribed disk of the largest radius. Essentially, we have shown that the set of centres of inscribed disks of the largest radius for $\Omg$ contains the origin. 
\end{proof}

\section{Optimisation of the second Robin eigenvalue on an exterior domain with fixed elastic energy of the boundary}\label{Sec.optim}
In this section, we establish a stronger version of Theorem~\ref{Thm.lambda2}.
In order to do so, we modify the strategy used in~\cite{EL22} 
to show an optimisation result for the second Robin eigenvalue
under a constraint on the maximum of the curvature. 

Let $\Omg\subset\dR^2$ be a bounded, convex domain with $C^\infty$-smooth boundary $\p\Omg$ of length $L > 0$. 
We parameterise the boundary of $\Omg$ in the clockwise direction by the unit-speed mapping $\s\colon [0,L]\arr\dR^2$ with $|\dot\s(s)| = 1$ for all $s\in[0,L]$. The unit tangential vector to the boundary is introduced by $\tau(s) := \dot\s(s) = (\tau_1(s),\tau_2(s))^\top$. The outer unit normal vector to $\Omg$ at the point $\s(s)$ is given by $\nu(s) = (-\tau_2(s),\tau_1(s))^\top$. 
We denote by $\kp\colon[0,L]\arr[0,\infty)$ the curvature of $\p\Omg$,
which we define by the Frenet formula 
\begin{equation}\label{eq:Frenet}
\dot\tau = -\kp \nu\,.
\end{equation}
The definition is fixed in such a way 
that~$\kappa$ is positive if~$\Omega$ is strictly convex.

According to~\cite[Section 4]{KL18}, the mapping
\[
(\dR/ (L\dZ))\times\dR_+\ni (s,t)\mapsto \s(s)+t\nu(s)
\]
is a diffeomorphism onto $\Omgc$. It defines parallel coordinates $(s,t)$ on $\Omgc$. 
Using the Frenet formula~\eqref{eq:Frenet}, 
one easily gets 
(see \eg~\cite[Section 4]{L23}) that for any $u\in L^2(\Omgc)$ 
\begin{equation}\label{eq:integral}
	\int_{\Omg^{\rm c}}|u|^2\dd x = \int_0^L\int_0^\infty |u(s,t)|^2 \, (1+\kp(s)t)\dd t\dd s,
\end{equation}
and that the gradient in the parallel coordinates $(s,t)$ is given by  \begin{equation}\label{eq:gradient}	
\nabla   = \frac{\tau(s)}{1+\kp(s) t}\p_s + \nu(s)\p_t\,.
\end{equation}

Recall that the \emph{elastic energy} of $\p\Omg$ is introduced by the formula
\begin{equation}\label{eq:elastic}
E(\p\Omg) := \frac12\int_0^L \kp^2(s)\dd s\,.
\end{equation}
We remark that the elastic energy can be defined by the same formula also for less regular $C^2$-smooth (and even $C^{1,1}$-smooth) not necessarily convex domains. 
The literature on the elastic energy of closed curves is quite extensive. For a detailed bibliography on elastic energies of closed curves in the plane, we refer to the
classical paper~\cite{LS85} or the more recent~\cite{Sa12}. An isoperimetric inequality for the elastic energy of the boundary of a convex domain under fixed area constraint was obtained by Gage in~\cite{Ga83}. Namely, he proved the inequality
\[
	E(\p\Omg) \ge \frac{\pi}{2}\frac{|\p\Omg|}{|\Omg|}
\]
for any bounded convex 
$C^2$-smooth domain $\Omg$, in which the equality is attained for disks. The above inequality is known to be false if the convexity condition is dropped. The inequality by Gage combined with the isoperimetric inequality yields as a consequence
\begin{equation}\label{eq:BH}
	E^2(\p\Omg)|\Omg| \ge \pi^3.
\end{equation}
Later it was shown in~\cite[Theorem 1.1]{BH17} that in the class of $C^\infty$-smooth simply-connected bounded domains the isoperimetric inequality~\eqref{eq:BH} holds without the convexity assumption on $\Omg$ and the equality is attained if, and only if, $\Omg$ is a disk. This inequality implies that among bounded simply-connected $C^\infty$-smooth domains of fixed area the disk is the unique minimiser of the elastic energy of the boundary.

\begin{remark}\label{rem:elastic_perimeter}
In this remark we discuss the relation between the perimeters of a general bounded simply-connected domain and the disk under the assumption that the elastic energies of their boundaries coincide.
Assume that a bounded $C^\infty$-smooth simply-connected  domain $\Omg\subset\dR^2$ (with curvature
of the boundary $\kp$ and of perimeter $L$) is not congruent to the disk $\cB\subset\dR^2$, which satisfies
\[
	E(\p\Omg) = E(\p \cB) = \frac{\pi}{R},
\]
where $R> 0$ is the radius of $\cB$. 
Under this geometric assumption, 
we get using the total curvature identity and the Cauchy--Schwarz inequality
\[
2\pi = \int_0^L\kp(s)\dd s < (2E(\p\Omg))^{1/2}\sqrt{L}\,,
\]
which implies that
\[	
L > \frac{2\pi^2}{E(\p\Omg)} =2\pi R = |\p \cB|\,.
\]
Thus, under fixed elastic energy of the boundary the disk is the unique minimiser of the perimeter among bounded simply-connected $C^\infty$-smooth domains. 
\end{remark}

Now we are ready to formulate and prove the second main result of the paper on 
the isoelastic inequality for the second Robin eigenvalue on exterior domains and a related geometric bound on the critical coupling constant.
\begin{thm}\label{thm2}
	Let $\Omg\subset\dR^2$ be a bounded convex $C^\infty$-smooth domain with the boundary $\p\Omg$ of length $L > 0$ and $\cB\subset\dR^2$ be the disk of radius $R >0$
	such that $E(\p\Omg) = E(\p \cB)$.
	Let $\kp\colon[0,L]\arr[0,\infty)$ be the curvature of $\p\Omg$ defined in the arc-length parametrisation.
	\begin{myenum}
	\item For $\Omg$ being not congruent to $\cB$, there holds
	\[
	\lm_2^\aa(\Omg^{\rm c}) < \lm_2^\aa(\cB^{\rm c}),\qquad\text{for all}\,\,\aa < -\frac{1}{R}.
	\]
	\item The critical coupling constant for $\Omgc$ satisfies
	\[
		\aa_\star(\Omgc) \ge-\frac{1}{2\pi}\int_0^L \kp^2(s)\dd s.
	\]
	\end{myenum}
\end{thm}
\begin{proof}
	We will prove both items simultaneously.
	For convenience of the reader the proof is divided into three steps. Throughout the proof
	of both items we assume without loss of generality that $\Omg$ is not congruent to $\cB$.
	
	\noindent {\it Step 1: trial functions.}
	In this step we construct two linearly independent trial functions
	in the Sobolev space $H^1(\Omg^{\rm c})$
	and provide estimates for them.
	
	It follows from Proposition~\ref{prop:disk}
	that there exist  non-trivial real-valued functions $f,g\in L^2(\dR_+; (R+t)\dd t)$ with $f',g'\in L^2(\dR_+;(R+t)\dd t)$ such that 
	\[
	\begin{aligned}
	\ker\big(-\Delta_\aa^{\cBc} - \lm_1^\aa(\cBc)\big) &= {\rm span}\,\{f(r-R)\},\\
	\ker\big(-\Delta_\aa^{\cBc} - \lm_2^\aa(\cBc)\big) &= {\rm span}\,\big\{g(r -R)e^{\ii\tt},g(r-R)e^{-\ii\tt}\big\}, 
	\end{aligned}
	\]
	where the functions in the spans are written in polar coordinates on $\cBc$.
	We remark that in terms of modified Bessel functions the functions $f$ and $g$ can be expressed as
	\[
		f(t) = K_0\left((-\lm_1^\aa(\cBc))^{\frac12}(t+R)\right)\quad\text{and}\quad g(t) = K_1\left((-\lm_2^\aa(\cBc))^{\frac12}(t+R)\right).
	\]
	However, the explicit expression for $f$ and $g$ will not be used in the proof. 
	Substituting the functions $f(r -R)$ and $g(r-R)e^{\ii\tt}$ into the Rayleigh quotient for $-\Delta_{\aa}^{\cBc}$, we get using the second variational characterisation in~\eqref{eq:minmax} that
	\begin{subequations}
	\begin{align}
		\label{eq:disk1}
		\lm_1^\aa(\cBc) &= \frac{\displaystyle\int_0^\infty|f'(t)|^2(R+t)\dd t + \aa R|f(0)|^2}{\displaystyle\int_0^\infty|f(t)|^2(R+t)\dd t} < 0,\\
				\label{eq:disk2}
		\lm_2^\aa(\cBc) &= \frac{\displaystyle\int_0^\infty|g'(t)|^2(R+t)\dd t +\int_0^\infty\frac{|g(t)|^2}{R+t}\dd t + \aa R|g(0)|^2}{\displaystyle\int_0^\infty|g(t)|^2(R+t)\dd t} < 0.
	\end{align}
	\end{subequations}
	Let us introduce two auxiliary functions on $\Omg^{\rm c}$ in the parallel coordinates $(s,t)$ by
	\[
	u_\star(s,t) := f(t)\quad\text{and}\quad
	v_\star(s,t) := g(t)\big[\tau_1(s) + \ii\tau_2(s)\big]\,.
	\]
	We obtain using~\eqref{eq:integral} that
	\begin{equation}\label{eq:L2norms}
	\begin{aligned}
	\int_{\Omg^{\rm c}}|u_\star|^2\dd x &= \int_0^\infty|f(t)|^2(L+2\pi t)\dd t,\\
	\int_{\Omg^{\rm c}}|v_\star|^2\dd x &= \int_0^\infty|g(t)|^2(L+2\pi t)\dd t.
	\end{aligned}
	\end{equation}  
	Moreover, we get combining~\eqref{eq:integral} with the expression for the gradient in~\eqref{eq:gradient} that
	\begin{equation}\label{eq:L2grad}
	\begin{aligned}
	\int_{\Omg^{\rm c}}|\nabla u_\star|^2\dd x &= \int_0^\infty |f'(t)|^2(L+2\pi t)\dd t,\\
	\int_{\Omg^{\rm c}} |\nabla v_\star|^2\dd x &= \int_0^\infty|g'(t)|^2(L+2\pi t) + \int_0^L
	\int_0^\infty\frac{\kp^2(s)}{1+t\kp(s)}|g(t)|^2\dd t \dd s\,,
	\end{aligned}
	\end{equation}
	where the Frenet formula~\eqref{eq:Frenet} was employed in the evaluation of the second integral. The identities~\eqref{eq:L2norms} and~\eqref{eq:L2grad} combined with the properties of the functions $f$ and $g$ (stated above) yield that $u_\star,v_\star\in H^1(\Omg^{\rm c})$.
	
	Consider the auxiliary function
	\[
	F_t(x) := \frac{x}{1+t \sqrt{x}}\,,\qquad t > 0\,. 
	\]
	Differentiating the function $F_t$ (with respect to $x$) twice we get
	\[
	F''_t(x) = -\frac{t(t\sqrt{x}+3)}{4\sqrt{x}(t\sqrt{x}+1)^3} < 0\,,\qquad\text{for all}\,\,x>0\,.
	\]
	Thus, the function $F_t$ is strictly concave
	and applying Jensen's inequality~\cite[Theorem 2.2]{LL01}
	taking into account that the curvature $\kp$ is not a constant function we get for any $t > 0$ 
	\[
	\begin{aligned}
	\frac{1}{L}\int_0^L \frac{\kp^2(s)}{1+t\kp(s)}\dd s&=\frac{1}{L}
	\int_0^L F_t(\kp^2(s))\dd s <  F_t\left(\frac{1}{L}\int_0^L\kp^2(s)\dd s\right)\\
	& = \frac{\frac{1}{L}\frac{2\pi}{R}}{1+t\sqrt{\frac{1}{L}}\sqrt{\frac{2\pi}{R}}}    
	=
	\frac{\frac{1}{\sqrt{L}}\frac{2\pi}{R}}{\sqrt{L}+t\sqrt{\frac{2\pi}{R}}}
	< 
	\frac{\frac{1}{\sqrt{2\pi R}}\frac{2\pi}{R}}{\sqrt{2\pi R}+t\sqrt{\frac{2\pi}{R}}}
	=\frac{1}{R}\cdot \frac{1}{R+t}\,,
	\end{aligned}
	\]
	where we used that $E(\p\Omg) = \frac{\pi}{R}$ for $R$ being the radius of $\cB$ and that $L > 2\pi R$ (see Remark~\ref{rem:elastic_perimeter}). As a consequence of this bound we obtain
	\begin{equation}\label{eq:grad_estimate}
	\int_{\Omg^{\rm c}}|\nabla v_\star|^2\dd x < \int_0^\infty|g'(t)|^2(L+2\pi t)\dd t +\frac{L}{R}\int_0^\infty\frac{|g(t)|^2}{R+t}\dd t\,.
	\end{equation}

	\noindent{\it Step 2: orthogonality.}
	In this step we show that $u_\star$ and $v_\star$ are orthogonal in $L^2(\Omg^{\rm c})$, that their gradients are orthogonal in $L^2(\Omg^{\rm c};\dC^2)$, and that their traces are orthogonal in $L^2(\p\Omg)$. 
	Using the Frenet formula~\eqref{eq:Frenet} we obtain that
	\[
	\begin{aligned}
	\int_{\Omg^{\rm c}} u_\star\ov{v_\star}\dd x 
	& =\int_0^L\int_0^\infty f(t)g(t)(\tau_1(s)-\ii\tau_2(s))(1+\kp(s)t)\dd t\dd s\\
	&=
	\int_0^\infty f(t)g(t)
	\left[\int_0^L(\dot\s_1(s)-\ii\dot\s_2(s))\dd s
	+t\int_0^L \big(\kp(s)\tau_1(s)-\ii \kp(s)\tau_2(s)\big)\dd s
	\right]\dd t \\
	&=-\int_0^\infty f(t)g(t)t\left[\int_0^L(\dot\tau_2(s) + \ii\dot\tau_1(s))\dd s\right]\dd t = 0\,.
	\end{aligned} 
	\]
	Combining the Frenet formula and the expression for the gradient in~\eqref{eq:gradient} and performing a computation similar to the one above we get
	\[
	\int_{\Omg^{\rm c}} \nabla u_\star\cdot\ov{\nabla v_\star}\dd x 
	=\int_0^L\int_0^\infty f'(t)g'(t)(\tau_1(s)-\ii\tau_2(s))(1+\kp(s)t)\dd t\dd s=0\,.
	\]
	Finally, we also obtain that
	\[
	\int_{\p\Omg} u_\star\ov{v_\star}\dd\s = f(0)g(0)\int_0^L
	\big(\tau_1(s)-\ii\tau_2(s)\big)\dd s =0\,.
	\]
	\noindent
	{\it Step 3: min-max principle.}
	Recall that the quadratic form $\frh_\aa^{\Omgc}$ for the self-adjoint operator $-\Delta^{\Omg^{\rm c}}_\aa$ acting in the Hilbert space $L^2(\Omg^{\rm c})$ is defined as in~\eqref{eq:form}.
	We introduce for $u\in H^1(\Omg^{\rm c})$, $u\ne0$, the notation for the Rayleigh quotient evaluated on $u$
	\[
	R[u] := \frac{\frh_\aa^{\Omg^{\rm c}}[u]}{\|u\|^2_{L^2(\Omg^{\rm c})}}.
	\] 
	Using the first identities in~\eqref{eq:L2norms} and in~\eqref{eq:L2grad} we get
	\begin{equation}\label{eq:R1}
	\begin{aligned}
	R[u_\star] &= \frac{\displaystyle
		\int_0^\infty|f'(t)|^2(L+2\pi t)\dd t +\aa L|f(0)|^2}{\displaystyle
		\int_0^\infty |f(t)|^2(L+2\pi t)\dd t}\\
	&=	
	\frac{\displaystyle
		\int_0^\infty|f'(t)|^2\left(R+\frac{2\pi R}{L} t\right)\dd t +\aa R|f(0)|^2}{\displaystyle
		\int_0^\infty |f(t)|^2\left(R+\frac{2\pi R}{L} t\right)\dd t}\\
	& <
	\frac{\displaystyle
		\int_0^\infty|f'(t)|^2(R+ t)\dd t +\aa R|f(0)|^2}{\displaystyle
		\int_0^\infty |f(t)|^2(R+t)\dd t} = \lm_1^\aa(\cBc) < \lm_2^\aa(\cBc),
	\end{aligned} 
	\end{equation}
	where we used in between $L > 2\pi R$ combined with~\eqref{eq:disk1}. 
		
	Employing the second identity in~\eqref{eq:L2norms} and the estimate~\eqref{eq:grad_estimate} we get
	\begin{equation}\label{eq:R2}
	\begin{aligned}
	R[v_\star] & < \frac{\displaystyle
		\int_0^\infty|g'(t)|^2(L+2\pi t)\dd t +\frac{L}{R}\int_0^\infty \frac{|g(t)|^2}{t+R}\dd t + \aa L|g(0)|^2}{\displaystyle\int_0^\infty|g(t)|^2(L+2\pi t)\dd t}\\
	&=\frac{\displaystyle
		\int_0^\infty|g'(t)|^2\left(R+\frac{2\pi R}{L} t\right)\dd t +\int_0^\infty \frac{|g(t)|^2}{t+R}\dd t + \aa R|g(0)|^2}{\displaystyle\int_0^\infty|g(t)|^2\left(R+\frac{2\pi R}{L} t\right)\dd t} \\
	&<
	\frac{\displaystyle
		\int_0^\infty|g'(t)|^2\left(R+ t\right)\dd t +\int_0^\infty \frac{|g(t)|^2}{t+R}\dd t + \aa R|g(0)|^2}{\displaystyle\int_0^\infty|g(t)|^2\left(R+ t\right)\dd t} = \lm_2^\aa(\cB^{
		\rm c})\,,
	\end{aligned}
	\end{equation}
	where we used in between $L > 2\pi R$ combined with~\eqref{eq:disk2}.
	
	For any $u = a u_\star + bv_\star\neq 0$ with $a,b\in\dC$ we get employing
	\eqref{eq:R1},~\eqref{eq:R2} and the orthogonality properties obtained in Step 2
	\[
	\begin{aligned}
	\frh_\aa^{\Omg^{\rm c}}[u] &= |a|^2\frh_\aa^{\Omg^{\rm c}}[u_\star]
	+|b|^2\frh_\aa^{\Omg^{\rm c}}[v_\star]\\
	& < \lm_2^\aa(\cB^{\rm c})\left(|a|^2\|u_\star\|^2_{L^2(\Omg^{\rm c})} + |b|^2\|v_\star\|^2_{L^2(\Omg^{\rm c})}\right) = \lm_2^\aa(\cB^{\rm c})\|u\|^2_{L^2(\Omg^{\rm c})}\,.
	\end{aligned}
	\]
	Thus, we conclude from the first characterisation in~\eqref{eq:minmax} that $\lm_2^\aa(\Omg^{\rm c}) < \lm_2^\aa(\cB^{\rm c})$ and the inequality in (i) is proved.
	
	The estimate on the critical coupling constant
	$\aa_\star(\Omgc)$ in (ii)  follows as a by-product.
	Indeed, the eigenvalue inequality in~(i)
	implies that the negative spectral subspace of the operator $-\Delta_\aa^{\Omgc}$
	is of dimension at least two for all 
	\[
		\aa < -\frac{1}{R} = -\frac{1}{\pi}E(\p\cB) = -\frac{1}{\pi}E(\p\Omg).
	\] 
	Hence, we conclude by the definition of the critical coupling constant combined with the expression~\eqref{eq:elastic} for $E(\p\Omg)$ that
	\[
	\aa_\star(\Omgc)\ge-\frac{1}{2\pi}\int_0^L\kp^2(s)\dd s.\qedhere
	\]
\end{proof}
\begin{remark}\label{Rem.EL22}
	In~\cite[Theorem 1.3]{EL22} the inequality
	$\lm_2^\aa(\Omgc)<\lm_2^\aa(\cBc)$ for all $\aa <-\frac{1}{R}$ was proved under the assumption that $\max\kp \le \frac{1}{R}$, provided that $\Omg\subset\dR^2$ is a bounded convex $C^\infty$-smooth domain, not congruent to the disk $\cB$ of radius $R>0$. In this remark we show that this result can be derived as a consequence of the inequality in Theorem~\ref{thm2}\,(i). Indeed, under the geometric assumption $\max\kp\le \frac{1}{R}$ we have
	\[
		E(\p\Omg) = \frac{1}{2}\int_0^L\kp^2(s)\dd s< \frac{1}{2R}\int_0^L \kp(s)\dd s = \frac{\pi}{R} = E(\p\cB),
	\]
	where we used the total curvature identity.
	Hence, there exists $R_\star >R$ such that
	$E(\p\Omg) =\frac{\pi}{R_\star}$. Let $\cB_\star$ be the disk of radius $R_\star > 0$. 
	By the isoelastic inequality in Theorem~\ref{thm2}\,(i) combined with monotonicity property for the disks in
	Proposition~\ref{prop:disk}\,(ii), 
	one has
	\[
		\lm_2^\aa(\Omgc)\le \lm_2^\aa(\cB_\star^{\rm c}) < \lm_2^\aa(\cBc),\qquad\text{for all}\,\,\,\aa<-\frac{1}{R}.
	\]
	We also point out that under the assumption $\max\kp\le \frac{1}{R}$ the inclusion $\cB\subset\Omg$ holds.
\end{remark}
\begin{remark}
	We emphasise that under the constraint $E(\p\Omg) = E(\p\cB)$ the inclusion $\cB\subset\Omg$, in general, does not hold
	(even in the class of convex domains).
	It was proved in~\cite{HM17} that among bounded convex $C^{1,1}$-smooth domains with fixed elastic energy of the boundary the in-radius is minimised by a domain different from the disk. In the same paper this optimal domain was explicitly characterised. In particular, under the constraint $E(\p\Omg) = \pi$ the smallest possible in-radius among bounded convex $C^{1,1}$-domains is given by the formula
	\[
		r_{\rm i}^\star = \frac{2}{\pi}\left(\int_0^{\frac{\pi}{2}}\sqrt{\cos t}\dd t\right)^2\approx 0.914 < 1.
	\]
	The radius of the disk with the elastic energy equal to $\pi$ is clearly $R =1$. Thus, this disk can not be contained in the minimiser of the in-radius.
\end{remark}
\subsection*{Acknowledgement}
The first author (D.K.)
was supported by the EXPRO grant 
No. 20-17749X of the Czech Science Foundation (GA\v{C}R).
The second author (V.L.) acknowledges the support by 
the grant No.~21-07129S of the Czech Science Foundation (GA\v{C}R).

%

\newcommand{\etalchar}[1]{$^{#1}$}

\end{document}